\documentclass[11pt, reqno]{amsart}
\usepackage{amssymb, amsthm, amsmath, amsfonts,mathtools}
\usepackage{array, epsfig}
\usepackage{bbm}

\usepackage{hyperref}
\usepackage{url}
\usepackage[numbers,square]{natbib}
\usepackage{color}
\allowdisplaybreaks

\numberwithin{equation}{section}
  \evensidemargin
\oddsidemargin
\newcommand{\ii}{{\rm{i}}}

\newcommand{\bE}{\mathbb{E}}

\newcommand{\bB}{\mathbb{B}}

\newcommand{\bS}{\mathbb{S}}

\newcommand{\bJ}{\mathbb{J}}

\def\dint{\textup{d}}

\newcommand{\E}{\mathbb E}

\newcommand{\R}{\mathbb{R}}
\newcommand{\N}{\mathbb{N}}

\newcommand{\C}{\mathbb{C}}

\renewcommand{\P}{\mathbb{P}}

\renewcommand{\Re}{\operatorname{Re}}

\newcommand{\pos}{\mathop{\mathrm{pos}}\nolimits}
\newcommand{\lin}{\mathop{\mathrm{lin}}\nolimits}

\newcommand{\ind}{\mathbbm{1}}

\newcommand{\dd}{{\rm d}}
\newcommand{\eee}{{\rm e}}

\theoremstyle{plain}
\newtheorem{theorem}{Theorem}[section]
\newtheorem{lemma}[theorem]{Lemma}

\newtheorem{proposition}[theorem]{Proposition}
\newtheorem{conjecture}[theorem]{Conjecture}

\theoremstyle{definition}

\newtheorem{example}[theorem]{Example}

\theoremstyle{remark}
\newtheorem{remark}[theorem]{Remark}

\makeatletter
\@namedef{subjclassname@2020}{\textup{2020} Mathematics Subject Classification}
\makeatother

\begin{document}

\author{Anna Gusakova}
\address{Anna Gusakova: Institut f\"ur Mathematische Stochastik,
Westf\"alische Wilhelms-Universit\"at M\"unster,
Orl\'eans-Ring 10,
48149 M\"unster, Germany}
\email{gusakova@uni-muenster.de}

\author{Zakhar Kabluchko}
\address{Zakhar Kabluchko: Institut f\"ur Mathematische Stochastik,
Westf\"alische Wilhelms-Universit\"at M\"unster,
Orl\'eans-Ring 10,
48149 M\"unster, Germany}
\email{zakhar.kabluchko@uni-muenster.de}

\title[Sylvester's problem for beta-type distributions]{Sylvester's problem for beta-type distributions}

\keywords{Sylvester problem, beta distribution, beta prime distribution, normal distribution, random polytope, random simplex, regular simplex, internal angle}

\subjclass[2020]{Primary: 60D05, 52A22; Secondary: 52A40}

\begin{abstract}
Consider $d+2$ i.i.d.\ random points $X_1,\ldots, X_{d+2}$ in $\R^d$.  In this note, we compute the probability that their convex hull is a simplex focusing on three specific distributional settings:
\begin{itemize}
\item[(i)] the distribution of $X_1$ is multivariate standard normal;
\item[(ii)] the density of $X_1$ is proportional to $(1-\|x\|^2)^{\beta}$ on the unit ball (the beta distribution);
\item[(iii)] the density of $X_1$ is proportional to $(1+\|x\|^2)^{-\beta}$ (the beta prime distribution).
\end{itemize}
In the Gaussian case, we show that this probability equals twice the sum of the solid angles of a regular $(d+1)$-dimensional simplex.
\end{abstract}

\maketitle

\section{Introduction}
Consider $n\geq d+1$ i.i.d.\ random points $X_1,\ldots, X_{n}$ in $\R^d$. Let $[X_1,\ldots, X_n]$ be their convex hull. Sylvester's  problem \cite{Syl64} consists of determining the probability that these points are in convex position meaning that all points $X_1,\ldots, X_n$ are vertices of the polytope $[X_1,\ldots, X_n]$.  This problem (originally stated by Sylvester for four points in the plane) has been intensively studied in the case when the points are uniformly distributed in a convex set; see the surveys in~\cite{BaranyBullSurvey}, \cite{CalkaSurvey}, \cite{hug_rev}, \cite{pfiefer_historical_development_sylvester},  \cite{reitzner_random_polys_survey}, \cite[pp.\ 63--65]{lS76}, \cite{schneider_discrete_aspects_stoch_geom}, \cite[Chapter~8]{schneider_weil_book}, \cite[Chapter 5]{Solomon1975}, as well as~ \cite{beck_phd, marckert_sylvester_n_disk_convex_pos, marckert_rahmani_around_sylvester_plane,morin_n_points_convex_position_regulag_k_gon} for some recent impressive results.

We shall be interested in the case when $n= d+2$ and the probability distribution of $X_1$, denoted by $\mu$, is absolutely continuous. The latter assumptions ensures that, with probability $1$,  the points are in general affine position meaning that no $d+1$ points are contained in a common affine hyperplane. Then, the convex hull $[X_1,\ldots, X_{d+2}]$ is either a simplex with $d+1$ vertices or a polytope with $d+2$ vertices, with probability $1$. The former case occurs if and only if one of the points falls into the convex hull of the remaining points, thus
\begin{align*}
p_d(\mu)&:=\P\big[[X_1,\ldots, X_{d+2}]\text{ is a simplex}\big] \\
&= (d+2)\cdot \P\big[X_{d+2} \in [X_1,\ldots, X_{d+1}]\big] =  (d+2) \cdot\E \mu ([X_1,\ldots, X_{d+1}]).
\end{align*}

It should be pointed out that computing $p_d(\mu)$ for general $\mu$ is basically out of reach. Even if $\mu$ is a uniform distribution in a cube or in a regular simplex, explicit formulas for $p_d(\mu)$ are available only in small dimensions:  see~\cite{Mann94} and~\cite{buchta_reitzner} for the uniform distribution in a tetrahedron, \cite{Zin03} for the three-dimensional cube and~\cite{Ali39, Buch84} for a planar regular $r$-gone.
The only explicit result valid in any dimension we are aware of is due to~\citet{kingman_secants} who has shown that if $\mu$ is a uniform distribution in a $d$-dimensional ball (or, more generally, in  a $d$-dimensional ellipsoid), then
\begin{equation}\label{eq:kingman_sylvester_simplex}
p_d(\mu)= \frac{d+2}{2^d} \cdot  \binom{d+1}{(d+1)/2}^{d+1} \cdot  \binom{(d+1)^2}{(d+1)^2/2}^{-1}.
\end{equation}

Another case which has been previously addressed in the literature is the case when $\mu_{\rm Gauss}$ is a $d$-dimensional standard Gaussian measure.  Maehara~\cite{maehara_on_sylvester_japanese} and Blatter~\cite{blatter_four_shots} solved Sylvester's problem for $4$ normally distributed points in the plane showing that they form a triangle with probability $1 - \frac {6}{\pi}\arcsin \frac 13$.
Recently, Frick, Newman and Pegden~\cite{frick_newman_pedgen} related Sylvester's problem for $d+2$ normally distributed points in $\R^d$ to the so-called Youden's demon problem studied previously by~\citet{kendall_two_problems_youden} and~\citet{david_sample_mean_among_extreme,david_sample_mean_moderate_order}. In particular, in \cite{frick_newman_pedgen} it was shown that $\frac 12 p_d(\mu_{\rm Gauss})$ coincides with the probability that in a one-dimensional i.i.d.\ standard normal sample $\xi_1,\ldots, \xi_{d+2}$ the mean $\frac{1}{d+2}(\xi_1+ \ldots + \xi_{d+2})$ lies between the first and the second smallest element of the sample. Moreover, in \cite{frick_newman_pedgen} it was also proven that  $5$ independent and normally distributed points in $\R^3$ form a simplex with probability $\frac 12 - \frac{5}{\pi} \arcsin \frac 14$.  In fact, the solution of the Sylvester problem for the Gaussian distribution (including the values for $d=2$ and $3$ mentioned above) can be easily deduced from~\cite{kabluchko_zaporozhets_absorption}, as we shall show below.  It should be also mentioned that even earlier Efron~\cite[Equations (7.5) and (7.6) on p.~341]{efron} obtained formulas for the expected probability content of a convex hull of $n$ normally distributed random points in dimensions $d=2,3$.

An analogue of the Sylvester problem for spherical convex hulls of $d+2$ points following the uniform distribution on the $d$-dimensional sphere has been solved by~\citet{maehara_martini_sylvester_sphere}; see also~\cite[Theorem~3.6]{maehara_martini_geometric_probab} and~\cite[Corollary~8.1]{maehara_martini_book} for low-dimensional cases. A similar problem on  the half-sphere has been solved in~\cite[Section~2.4]{kabluchko_poisson_zero}.

\vspace*{2mm}
The aim of the present note is to solve the Sylvester problem for $n=d+2$ points following a beta-type distribution on $\R^d$. The family of beta-type distributions, introduced in \cite{miles} and~\cite{ruben_miles}, includes the $d$-dimensional standard normal distribution, the beta distributions with density proportional to $(1-\|x\|^2)^{\beta}\ind_{\{\|x\| <  1\}}$ (where $\beta>-1$ is a parameter) and the beta prime distributions with density proportional to $(1+\|x\|^2)^{-\beta}$ (where $\beta > d/2$ is a parameter). We provide a purely geometric argument reducing the Sylvester problem to computing the expected angle sums of beta-type simplices; to be defined in~\eqref{eq:Jbeta} and~\eqref{eq:Jbetaprime}, below. These angle sums have been determined in~\cite{kabluchko_angles_explicit_formula}. In the case of the normal distribution, our argument relates the Sylvester problem to the solid angles of the regular simplex. In particular, the values $1 - \frac {6}{\pi}\arcsin \frac 13$ and $\frac 12 - \frac{5}{\pi} \arcsin \frac 14$ mentioned above coincide with twice the angle sums of a regular tetrahedron and a regular $4$-dimensional simplex, respectively, -- it seems that this fact has been overlooked by the previous authors.


In Section~\ref{sec:results}, we state our main results concerning the Sylvester problem. The corresponding proofs are provided in Section~\ref{sec:proofs}.

\section{Results}  \label{sec:results}

\subsection{Notation}
Let $\bB^d:= \{x\in \R^d: \|x\| \leq  1\}$ be the unit ball and $\bS^{d-1}= \{x\in \R^d : \|x\| = 1\}$ be the unit sphere in $\R^d$. Here, $\|\cdot\|$ is the Euclidean norm.  The positive hull of $n$ vectors $x_1,\ldots, x_n\in \R^d$ is denoted by
$$
\pos(x_1,\ldots,x_n) := \left\{\lambda_1 x_1 + \ldots + \lambda_n x_n: \lambda_1,\ldots,\lambda_n\geq 0\right\}.
$$
A polyhedral cone is a positive hull of finitely many vectors in $\R^d$. Let $C\subseteq \R^d$ be a polyhedral cone, $\lin C$ the minimal linear subspace containing $C$ (called the linear hull of $C$), and $\lambda_{\lin C}$ the Lebesgue measure on $\lin C$.  The (solid) angle of $C$ is defined as
$$
\alpha(C) = \frac{\lambda_{\lin C}(\bB^d \cap C)}{\lambda_{\lin C} (\bB^d \cap \lin C)},
$$
which is the proportion of the unit ball occupied by the cone with $\lin C$ viewed as the ambient space. Note that $\alpha(C)$ takes values between $0$ and $1$. For more information on probabilistic aspects of convex cones we refer to the book of~\citet{schneider_book_convex_cones_probab_geom}.

Recall that given $n\ge d+1$ i.i.d.\ random points $X_1,\ldots,X_n$ in $\R^d$ with absolutely continuous distribution $\mu$ we denote by $[X_1,\ldots, X_n]$ their convex hull and by
\[
p_d(\mu):=\P\big[[X_1,\ldots, X_{d+2}] \text{ is a simplex}\big] =  (d+2) \cdot\E \mu ([X_1,\ldots, X_{d+1}]),
\]
the probability that they form a simplex.

\subsection{Sylvester problem for the normal distribution}
Let $X_1,\ldots, X_{d+2}$ be i.i.d.\ random points in $\R^d$ with standard normal distribution $\mu_{\rm Gauss}$. Our aim is to determine
$$
p_d(\mu_{\rm Gauss})=\P\big[\,[X_1,\ldots, X_{d+2}] \text{ is a simplex}\,\big].
$$

Before we state the result let us recall some facts about the regular simplex. Let $e_1,\ldots,e_n$ be the standard orthonormal basis in $\R^n$ and consider the $(n-1)$-dimensional regular simplex
$
\Delta_n:= [e_1,\ldots,e_n].
$
The internal angle of $\Delta_n$ at vertex $e_i$ is the angle of the cone $\pos (\{e_j-e_i\colon 1\leq j\leq n, j\neq i\})$. Since all these angles are equal, the sum of internal angles of $\Delta_n$ at its vertices is defined as
$$
\bJ_{n,1}(\infty): = n \cdot  \alpha\big(\pos (e_2-e_1,\ldots, e_n-e_1)\big).
$$
More generally, it is possible to define internal angle sums at $(k-1)$-dimensional faces of $\Delta_n$, for all $k\in \{1,\ldots, n\}$. These angle sums, denoted by $\bJ_{n,k}(\infty)$ to ensure consistency with~\cite{kabluchko_angles_explicit_formula}, will be needed here only in the special case $k=1$. The following explicit formula is known:
\begin{align}\label{eq:J_n_1_formula}
\bJ_{n,1}(\infty)
&=
n \cdot \frac 1 {\sqrt {2\pi}} \int_{-\infty}^{+\infty} \Phi^{n-1} \left( \frac{\ii x}{\sqrt n}\right)  \eee^{-x^2/2} \dd x,
\end{align}
where $\Phi(z)$ is cumulative distribution function of the standard normal random variable. Note that $\Phi(z)$ can be extended to an analytic function on the entire complex plane:
$$
\Phi(z)
=
\frac 12 + \frac 1 {\sqrt{2\pi}} \int_{0}^z \eee^{-t^2/2} \dd t
=
\frac 12  + \frac 1 {\sqrt{2\pi}} \sum_{n =0}^{\infty}  \frac{(-1)^n }{(2n+1) 2^n n!} z^{2n+1}, \qquad z\in \C.
$$
Formula~\eqref{eq:J_n_1_formula} goes back to Rogers~\cite[Section~4]{rogers} (where the method used was attributed to H.\ E.\ Daniels) and Vershik and Sporyshev~\cite[Lemma~4]{vershik_sporyshev_asymptotic_faces_random_polyhedra1992}; see also~\cite[Proposition~1.2]{kabluchko_zaporozhets_absorption} for this formula and further results.

\begin{theorem}[Sylvester problem for the normal distribution]\label{theo:sylvester_normal}
Let $X_1,\ldots, X_{d+2}$ be i.i.d.\ following a standard normal distribution on $\R^d$. Then, the probability that $[X_1,\ldots, X_{d+2}]$ is a simplex is given by
\begin{align}
p_d(\mu_{\rm Gauss})=2 \cdot \bJ_{d+2,1} (\infty) =  \frac {2(d+2)} {\sqrt {2\pi}} \int_{-\infty}^{+\infty} \Phi^{d+1} \left( \frac{\ii x}{\sqrt {d+2}}\right)  \eee^{-x^2/2} \dd x.
\label{eq:sylvester_normal}
\end{align}
\end{theorem}

\begin{example}
The sum of internal angles at vertices of the $3$- and $4$-dimensional simplices are known to be $\bJ_{4,1}(\infty) = \frac 12 - \frac {3}{\pi}\arcsin \frac 13$ and $\bJ_{5,1}(\infty) = \frac 14 - \frac{5}{2\pi} \arcsin \frac 14$ (see, e.g.,  \cite[Propositions~1.2 and~1.4 (f)]{kabluchko_zaporozhets_absorption}), giving the values of the Sylvester probability in dimensions $d=2$ and $d=3$:
$$
p_2(\mu_{\rm Gauss}) = 1 - \frac {6}{\pi}\arcsin \frac 13,
\qquad
p_2(\mu_{\rm Gauss}) = \frac 12 - \frac{5}{\pi} \arcsin \frac 14,
$$
which were previously obtained in~\cite{maehara_on_sylvester_japanese}, \cite{blatter_four_shots} and~\cite{frick_newman_pedgen}.
\end{example}

\begin{remark}
It is well known that the Sylvester problem is invariant under any nonsingular affine transformation $A:\R^d\to\R^d$ in the sense that  $[X_1,\ldots, X_{d+2}]$ is a simplex if and only if $[A X_1,\ldots, A X_{d+2}]$ is a simplex. It follows that Theorem~\ref{eq:sylvester_normal} provides a solution to Sylvester problem also if $X_1,\ldots, X_{d+2}$ are i.i.d.\ with any nonsingular multidimensional normal distribution on $\R^d$.
\end{remark}

\subsection{Sylvester problem for beta and beta prime distributions}
A probability measure $\mu_{d,\beta}$ is called a $d$-dimensional \textit{beta distribution} with parameter $\beta>-1$ if it has a density of the form
\[
f_{d,\beta}(x)=c_{d,\beta} \big( 1-\left\| x \right\|^2 \big)^\beta\ind_{\{\|x\| <  1\}},\qquad
c_{d,\beta}= \frac{ \Gamma\left( \frac{d}{2} + \beta + 1 \right) }{ \pi^{ \frac{d}{2} } \Gamma\left( \beta+1 \right) },
\]
with respect to the Lebesgue measure in $\R^d$. Although the function $(1-\| x \|^2 )^\beta\ind_{\{\|x\| <  1\}}$ is not integrable for $\beta=-1$ and hence $f_{d,-1}$ does not exist as probability density, the measure $\mu_{d,\beta}$ converges weakly as $\beta\downarrow-1$ to the uniform distribution on the unit sphere $\bS^{d-1}$. For this reason we will set $\mu_{d,-1}$ to be the  uniform distribution on $\bS^{d-1}$. Analogously, a probability measure $\tilde\mu_{d,\beta}$ is called $d$-dimensional \textit{beta prime distribution} with parameter $\beta>d/2$ if it has a density of the form
\[
\tilde{f}_{d,\beta}(x)=\tilde{c}_{d,\beta} \big( 1+\left\| x \right\|^2 \big)^{-\beta},\qquad
\tilde{c}_{d,\beta}= \frac{ \Gamma\left( \beta \right) }{\pi^{ \frac{d}{2} } \Gamma\left( \beta - \frac{d}{2} \right)},
\]
with respect to the Lebesgue measure in $\R^d$. Basic properties of beta and beta prime distributions are reviewed in~\cite{beta_polytopes} and~\cite{kabluchko_steigenberger_thaele_boob_beta_type}.
Let $X_1,\ldots, X_{d+2}$ be i.i.d.\ with the beta distribution $\mu_{d,\beta}$ and $\tilde X_1,\ldots, \tilde X_{d+2}$ i.i.d.\ with the beta prime distribution $\tilde\mu_{d,\beta}$. Our aim is to determine the probabilities
\begin{align*}
p_d(\beta)&:=p_d(\mu_{d,\beta})= \P\big[[X_1,\ldots, X_{d+2}] \text{ is a simplex}\,\big],
\\
\tilde p_d(\beta)&:=p_d(\tilde\mu_{d,\beta})=  \P\big[[\tilde X_1,\ldots, \tilde X_{d+2}] \text{ is a simplex}\,\big].
\end{align*}

We shall express these probabilities through the expected angle sums of beta and beta prime simplices which were computed in~\cite{kabluchko_angles_explicit_formula}. Let us recall some results from~\cite{beta_polytopes} and~\cite{kabluchko_angles_explicit_formula}.  We use the notation
\[
c_{\beta}
:=
c_{1,\beta}
=
\frac{ \Gamma\left(\beta + \frac{3}{2} \right) }{  \sqrt \pi\, \Gamma (\beta+1)},
\qquad
\tilde c_{\beta}
:=
\tilde{c}_{1,\beta}
=
\frac{ \Gamma(\beta)}{ \sqrt \pi\, \Gamma\left( \beta - \frac{1}{2}\right)}.
\]
Fix some  $n\geq 2$.  Let $Z_1,\ldots,Z_{n}$ be independent random points in $\R^{n-1}$ sampled from the beta distribution $\mu_{n-1,\beta}$, where $\beta \geq  -1$.   The convex hull $[Z_1,\ldots,Z_n]$ is called the $(n-1)$-dimensional \textit{beta simplex}. The internal angle of this simplex at its vertex $Z_i$, $1\leq i\leq n$, is the angle of the cone $\pos (\{Z_j-Z_i\colon 1\leq j\leq n, j\neq i\})$. All these internal angles are random and have the same distribution. The expected internal angle sum of $[Z_1,\ldots,Z_n]$ at its vertices is then denoted by
\begin{equation}\label{eq:Jbeta}
\bJ_{n,1}(\beta) := n \cdot  \E \alpha\big(\pos (Z_2-Z_1,\ldots, Z_n-Z_1)\big).
\end{equation}
It is known~\cite[Section~3.3.1]{kabluchko_angles_explicit_formula} that $\bJ_{n,1}(\beta) \to \bJ_{n,1}(\infty)$ as $\beta\to+\infty$. In fact, it is possible~\cite{beta_polytopes,kabluchko_angles_explicit_formula} to define expected internal angle sums $\bJ_{n,k}(\beta)$ at $(k-1)$-dimensional faces for all $k\in \{1,\ldots, n\}$, but we shall need only the case $k=1$ here.
An explicit formula for $\bJ_{n,k}(\beta)$ has been derived in~\cite[Equations~(1.5)--(1.7)]{kabluchko_angles_explicit_formula}. For $k=1$ it takes the form
\begin{equation}\label{eq:J_nk_integral}
\bJ_{n,1}\left(\frac{\alpha - n + 1}{2}\right)
=
n\cdot  \int_{-\infty}^{+\infty} c_{\frac{\alpha n}2} (\cosh x)^{- \alpha n - 2}
\left(\frac 12  + \ii \int_0^x  c_{\frac{\alpha-1}{2}} (\cosh y)^{\alpha}\dd y \right)^{n-1} \dd x,
\end{equation}
for all integer $n\geq 3$ and $\alpha \geq n-3$.

Similar situation appears in the beta prime case. Let $\tilde Z_1,\ldots,\tilde Z_{n}$ be independent random points in $\R^{n-1}$ sampled from the beta prime distribution $\tilde \mu_{n-1,\beta}$, where $\beta > (n-1)/2$. Their convex hull $[\tilde Z_1,\ldots,\tilde Z_n]$ is called the $(n-1)$-dimensional \textit{beta prime simplex} and its expected internal angle sum at vertices is denoted by
\begin{equation}\label{eq:Jbetaprime}
\tilde \bJ_{n,1}(\beta)  := n \cdot  \E \alpha\big(\pos (\tilde Z_2- \tilde Z_1,\ldots, \tilde Z_n - \tilde Z_1)\big).
\end{equation}
An explicit formula for the quantities $\tilde \bJ_{n,1}(\beta)$ has been derived in~\cite[Equations~(1.14)--(1.16)]{kabluchko_angles_explicit_formula}:
\begin{equation}\label{eq:J_nk_tilde_integral}
\tilde \bJ_{n,1}\left(\frac{\alpha + n - 1}{2}\right)
=
n\cdot  \int_{-\infty}^{+\infty} \tilde c_{\frac{\alpha n}2} (\cosh x)^{-(\alpha n-1)} \left(\frac 12  + \ii \int_0^x \tilde c_{\frac{\alpha+1}{2}}(\cosh y)^{\alpha-1}\dd y \right)^{n-1} \dd x,
\end{equation}
for all $n\in\N$ and $\alpha > 1/n$. It is known~\cite[Section~3.3.1]{kabluchko_angles_explicit_formula} that $\tilde \bJ_{n,1}(\beta) \to \bJ_{n,1}(\infty)$ as $\beta\to+\infty$.
Now we are prepared to state the solution of the Sylvester problem for  beta and beta prime distributions.
\begin{theorem}[Sylvester problem for beta and beta prime distributions]\label{theo:sylvester_beta}
Let $X_1,\ldots, X_{d+2}$ be i.i.d.\ with the beta distribution $\mu_{d,\beta}$ for some $\beta \geq -1/2$. Then, the probability that $[X_1,\ldots, X_{d+2}]$ is a simplex is given by
\begin{align}
p_d(\beta)
&=
2 \cdot \bJ_{d+2,1} \left(\beta - \frac 12\right) \label{eq:sylvester_beta_1}\\
&=
(2d+4) \int_{-\infty}^{+\infty} c_{\frac{(2\beta + d) (d+2)}2} (\cosh x)^{-(2\beta + d)(d+2) - 2}
\left(\frac 12  + \ii \int_0^x  c_{\frac{2\beta + d-1}{2}} (\cosh y)^{2\beta + d}\dd y \right)^{d+1} \dd x.
\label{eq:sylvester_beta_2}
\end{align}
In fact, formula~\eqref{eq:sylvester_beta_2} holds for $\beta\geq -1$ provided $d\geq 2$.
Similarly, if $\tilde X_1,\ldots, \tilde X_{d+2}$ are i.i.d.\ with the beta prime distribution $\tilde \mu_{d,\beta}$ for some $\beta >d/2$, then the probability that $[\tilde X_1,\ldots, \tilde X_{d+2}]$ is a simplex is given by
\begin{align*}
\widetilde p_d(\beta)
&=
2 \cdot \tilde \bJ_{d+2,1} \left(\beta + \frac 12\right) 
\\
&=
(2d+4) \int_{-\infty}^{+\infty} \tilde c_{\frac{(2\beta - d) (d+2)}2} (\cosh x)^{-((2\beta-d) (d+2)-1)} \left(\frac 12  + \ii \int_0^x \tilde c_{\frac{2\beta - d + 1}{2}}(\cosh y)^{2\beta - d - 1}\dd y \right)^{d+1} \dd x,
\end{align*}
where for the second formula we additionally require $2\beta > d + (d+2)^{-1}$ (otherwise, the integral does not converge).
\end{theorem}


Let us consider some special cases of Theorem~\ref{theo:sylvester_beta} in which the formulas for $p_d(\beta)$ and $\tilde p_d(\beta)$ simplify.

\begin{example}[Uniform distribution]
The beta distribution with $\beta=0$ is the uniform distribution on the unit ball. The probability that $[X_1,\ldots, X_{d+2}]$ is a simplex is given by  Kingman's formula~\eqref{eq:kingman_sylvester_simplex}, which coincides with expression for $2 \cdot \bJ_{d+2,1}(-1/2)$ obtained in~\cite[Theorem 3.9]{kabluchko_recursive_scheme}.
\end{example}
\begin{example}[Beta distribution with $\beta = 1$]
If $X_1,\ldots, X_{d+2}$ are i.i.d.\ random points\ in the unit ball $\bB^d$ with density proportional to $(1-\|x\|^2)$, then the Sylvester probability is
$$
p_d(1) =   2 \cdot \bJ_{d+2,1} \left(\frac 12\right) = \frac{2\pi (d+2)\left((d+2)^2+1\right)\left((d+2)^2 + d + 4\right)}{(d+5)\, 2^{(d+2)(2d+5)}} \binom{d+3}{\frac{d+3}{2}}^{d+1} \binom{(d+2)^2}{\frac {(d+2)^2}{2}},
$$
where the last identity follows from~\cite[Theorem~3.8]{kabluchko_recursive_scheme}.
\end{example}
\begin{example}[Arcsine distribution]
Let $X_1,\ldots, X_{d+2}$ be i.i.d.\ random points\ in the unit ball $\bB^d$ with density proportional to $1/\sqrt{1-\|x\|^2}$. Then, the Sylvester probability is $p_d(-1/2) =  2 \cdot \bJ_{d+2,1} (-1)$. The right-hand side can be evaluated in small dimensions, see~\cite[Section~3.1]{kabluchko_recursive_scheme}, giving
\begin{align*}
p_2\Big(-{1\over 2}\Big) &= {1\over 4}, &p_3\Big(-{1\over 2}\Big) &= \frac{539}{144 \pi ^2}-\frac{1}{3},\\
 p_4\Big(-{1\over 2}\Big) &= \frac{25411}{3670016},
&p_5\Big(-{1\over 2}\Big) &= \frac{1}{3}+\frac{113537407}{24192000 \pi ^4}-\frac{2144238917}{570810240 \pi ^2}.
\end{align*}
\end{example}

\begin{example}[Semispherical distribution]
Let $X_1,\ldots, X_{d+2}$ be i.i.d.\ random points in the unit ball $\bB^d$ with density proportional to $\sqrt{1-\|x\|^2}$. Then, the Sylvester probability is $p_d(1/2) =  2 \cdot \bJ_{d+2,1} (0)$. The table of values of $\bJ_{n,1}(0)$ given in  \cite[Section~3.2]{kabluchko_recursive_scheme} yields
$$
p_2\Big({1\over 2}\Big) = \frac{401}{1280}, \quad p_3\Big({1\over 2}\Big) = \frac{1692197}{423360 \pi ^2}-\frac{1}{3}, \quad p_4\Big({1\over 2}\Big) =\frac{112433094897}{8598524526592}.
$$
\end{example}

\begin{example}[Cauchy distribution]
Let $X_1,\ldots, X_{d+2}$ be i.i.d.\ random points in $\R^d$ with density proportional to $(1+\|x\|^2)^{-(d+1)/2}$.   A formula for the Sylvester probability has been derived in~\cite[Theorem~2.6]{kabluchko_poisson_zero}.
In~\cite[p.~402]{kabluchko_poisson_cells_high_dimensional} it has been shown that $p_d((d+1)/2) \sim 2\sqrt 3 d \pi^{-d-1}$ as $d\to\infty$.
In these articles, the results are stated in terms of convex cones in the upper half-space of $\R^{d+1}$. To pass to the Cauchy distribution, one can use the  gnomonic projection as described in~\cite[Section~2.2.1]{convex_hull_sphere}.
\end{example}

\begin{example}[Beta prime distribution with $\beta = (d/2) +1$]
Let $X_1, \ldots, X_{d+2}$ be i.i.d.\ random points\ in $\R^d$ with density proportional to $(1+\|x\|^2)^{-(d+2)/2}$. The Sylvester probability is
$$
\tilde p_d\left(\frac{d}{2} + 1\right) = 2 \cdot \tilde \bJ_{d+2,1} \left(\frac {d+3}2\right) = \frac{4 (2d+3)}{\binom{2d+4}{d+2}},
$$
where the last formula follows from~\cite[Theorem~2.6]{kabluchko_angles_explicit_formula}.
\end{example}

\subsection{Further results and conjectures}
If $\beta$ is integer or half-integer, it is possible to describe the arithmetic nature of the numbers $p_d(\beta) = 2 \cdot \bJ_{d+2,1} (\beta - \frac 12)$ and $\tilde p_d(\beta) = 2 \cdot \tilde \bJ_{d+2,1} (\beta + \frac 12)$. The next two results follow from Theorem~1.6 and Theorem~1.10 in~\cite{kabluchko_angles_explicit_formula}.
\begin{proposition}
\label{prop:arithm_sylvester_beta}
Let $\beta \ge -1/2$ be integer or half-integer and $d\in\N$.
\begin{itemize}
\item[(a)] If $2\beta + d$ is odd, then $p_d(\beta)$ is a rational number.
\item[(b)] If both $2\beta$ and $d$ are even, then $p_d(\beta)$ is a number of the form $q\pi^{-d}$ with some rational $q$.
\item[(c)] If both $2\beta$ and $d$ are odd, then $p_d(\beta)$ can be expressed as  $q_0 + q_2 \pi^{-2} + q_4 \pi^{-4} + \ldots + q_{d+1} \pi^{-d-1}$,
     where the numbers $q_{2j}$ are rational.
\end{itemize}
\end{proposition}

\begin{proposition}
\label{theo:arithm-sylvester_beta_prime}
Let $d\in\N$ and  $\beta> d/2$ be integer or half-integer.
\begin{itemize}
\item[(a)] If $2\beta - d$ is even, then $\tilde p_d(\beta)$ is a rational number.
\item[(b)] If $2\beta - d$ is odd, then $\tilde p_d(\beta)$ can be expressed as
$q_0 + q_2 \pi^{-2} + q_4 \pi^{-4} + \ldots + q_{d} \pi^{-d}$ (if $d$ is even) or
$q_0 + q_2 \pi^{-2} + q_4 \pi^{-4} + \ldots + q_{d+1} \pi^{-d-1}$ (if $d$ is odd),
where the coefficients $q_{2j}$ are rational numbers.
\end{itemize}
\end{proposition}

Numerical experiments strongly suggest  the following
\begin{conjecture}
Let $d\geq 2$. The Sylvester probability $p_d(\beta)$ is strictly increasing in $\beta>-1$ for all $d\geq 2$, while $\tilde p_d(\beta)$ is strictly decreasing in $\beta>d/2$.
\end{conjecture}

It can be shown~\cite[Section~3.3.1]{kabluchko_angles_explicit_formula} that $\bJ_{n,k}(\beta) \to \bJ_{n,k}(\infty)$ and $\tilde \bJ_{n,k}(\beta) \to \bJ_{n,k}(\infty)$ as $\beta\to+\infty$. Consequently, both $p_d(\beta)$ and $\tilde p_d(\beta)$ converge to $p_d(\mu_\text{Gauss})$ as $\beta \to +\infty$. Also, as $\beta\downarrow -1$, the beta distribution $\mu_{d,\beta}$ converges weakly to the uniform distribution on the unit sphere, for which the probability of a simplex is $0$. It follows that $\lim_{\beta\downarrow -1} p_d(\beta) = 0$.

\section{Proofs}\label{sec:proofs}
The proofs of both Theorem \ref{theo:sylvester_normal} and Theorem \ref{theo:sylvester_beta} are of geometric nature and rely on the following crucial observation. Given some $u\in \bS^{n-1}$ we denote by $u^{\perp}$ the hyperplane orthogonal to $u$. Given a hyperplane $H\subset\R^n$ denote by $\Pi_{H}:\R^n\mapsto H$ the orthogonal projection onto $H$. Then we have the following relation between the solid angles of a simplex and the probability that its uniform random projection is again a lower-dimensional simplex.

\begin{lemma}\label{lm:angles_probabilities}
Let $[x_1,\ldots,x_{n+1}]\subset \R^n$ be an $n$-dimensional simplex and let $U$ be a random vector uniformly distributed on the unit sphere $\bS^{n-1}$. Then
\[
\P\big[\Pi_{U^{\perp}}(x_{n+1})\in [\Pi_{U^{\perp}}(x_1),\ldots, \Pi_{U^{\perp}}(x_n)]\big]=2\alpha\big(\pos(x_1-x_{n+1},\ldots,x_{n}-x_{n+1})\big).
\]
\end{lemma}

\noindent This formula can be found for example in \cite[Equation (1)]{feldman_klain}, but the idea of interpreting probabilities as angles can be traced back to~\citet{shephard_gram_elementary} and~\citet{perles_shephard}. It is also used in~\cite{affentranger_schneider_random_proj}.

\begin{proof}[Proof of Theorem \ref{theo:sylvester_normal}]
Consider $d+2$ i.i.d.\ random points $Y_1,\ldots, Y_{d+2}$ in $\R^{d+1}$ with the standard normal distribution. Their convex hull $[Y_1,\ldots, Y_{d+2}]$ is a $(d+1)$-dimensional simplex a.s.  Let $H\subset \R^{d+1}$ be any hyperplane passing through $0$. Then the random points  $\Pi_H(Y_1),\ldots, \Pi_H(Y_{d+2})$ are independent and standard normal in $H$. (Note that after applying a rotation, we may assume that $H = \R^d$. This is possible since standard normal distribution is rotationally invariant.)
It follows that
$$
p_d(\mu_{\text{Gauss}})
=
(d+2)\cdot \P\big[X_{d+2} \in [X_1,\ldots, X_{d+1}]\big]
=
(d+2)\cdot \P\big[\Pi_H(Y_{d+2})\in [\Pi_H(Y_1),\ldots, \Pi_H(Y_{d+1})]\big],
$$
which holds for any hyperplane $H\subset\R^{d+1}$. By the formula of total probability, this conclusion remains in force if we choose $H= U^\perp$, where $U$ is uniformly distributed on $\bS^d$ and independent of $Y_1,\ldots,Y_{d+2}$. Then
\begin{align*}
p_d(\mu_{\text{Gauss}})&=(d+2)\cdot \bE\big[\P\big[\Pi_{U^{\perp}}(Y_{d+2})\in [\Pi_{U^{\perp}}(Y_1),\ldots, \Pi_{U^{\perp}}(Y_{d+1})]\,|\, Y_1,\ldots, Y_{d+2}\big]\big]\\
& =2(d+2)\cdot  \E \alpha \big(\pos (Y_1-Y_{d+2}, \ldots, Y_{d+1} - Y_{d+2})\big),
\end{align*}
where the second equality follows by Lemma \ref{lm:angles_probabilities}.
It has been shown in~\cite{kabluchko_zaporozhets_gauss_simplex} and~\cite{GoetzeKabluchkoZaporozhets} that the expected angle sums of the Gaussian simplex $[Y_1,\ldots, Y_{d+2}]$ coincide with the angle sums of the regular simplex of the same dimension, in particular
$$
(d+2) \cdot \E \alpha \big(\pos (Y_1-Y_{d+2}, \ldots, Y_{d+1} - Y_{d+2})\big) =  (d+2) \cdot \alpha \big(\pos(e_2 - e_{1}, \ldots , e_{d+2} - e_{1})\big) = \bJ_{d+2,1} (\infty),
$$
where $e_1,\ldots, e_{d+2}$ is the standard orthogonal basis in $\R^{d+2}$. Combining everything together proves that $p_d(\mu_{\text{Gauss}}) = 2 \cdot \bJ_{d+2,1} (\infty)$. To complete the proof of \eqref{eq:sylvester_normal} it remains to apply~\eqref{eq:J_n_1_formula}.
\end{proof}

\begin{remark}
Let us explain how to deduce Theorem~\ref{theo:sylvester_normal} from~\cite{kabluchko_zaporozhets_absorption}.
Let $\sigma^2 >0$. By Theorem~1.1 in~\cite{kabluchko_zaporozhets_absorption} (in which we take $n=d+1$,  consider the complementary probability and use Remark~1.1 in~\cite{kabluchko_zaporozhets_absorption} to simplify it), we have
\begin{align*}
\P[\sigma X_{d+2} \in [X_1,\ldots, X_{d+1}]]
&=
1-2(b_{d+1, d-1} (\sigma^2) + b_{d+1, d-3} (\sigma^2) + \ldots)
=
2 b_{d+1, d+1} (\sigma^2)
\\
&= \frac{2}{\sqrt{2\pi}}  \int_{-\infty}^{+\infty} \Phi^{d+1} \left( \frac{\ii x \sigma}{\sqrt {1+(d+1) \sigma^2}}\right)  \eee^{-x^2/2} \dd x,
\end{align*}
where the last formula follows from  Equation~(1.7) in~\cite{kabluchko_zaporozhets_absorption}.
Recalling that $p_d(\mu_{\text{Gauss}})=(d+2)\cdot \P\big[X_{d+2} \in [X_1,\ldots, X_{d+1}]\big]$ and taking $\sigma=1$ in the above formula gives~\eqref{eq:sylvester_normal}.
\end{remark}

\begin{proof}[Proof of Theorem \ref{theo:sylvester_beta}]
Let us consider the beta case first. Our aim is to show that $p_d(\beta) = 2 \cdot \bJ_{d+2,1} (\beta - \frac 12)$.
To this end, we consider $d+2$ i.i.d.\ random points $Y_1,\ldots, Y_{d+2}$ in $\R^{d+1}$ with distribution $\mu_{d+1,\beta - \frac 12}$. Note that we need the condition $\beta \geq - \frac 12$ to ensure that $\beta - \frac 12 \geq -1$.  Let $H\subseteq \R^{d+1}$ be any hyperplane passing through $0$. By the projection property of the beta distributions (see, e.g., \cite[Lemma~4.4]{beta_polytopes_temesvari} or~\cite[Lemma~3.1]{beta_polytopes}) the random points  $\Pi_H(Y_1),\ldots, \Pi_H(Y_{d+2})$ are independent and distributed according to $\mu_{d,\beta}$ in $H$. (To see this, we may assume that $H=\R^d$ by applying some rotation and using rotational invariance of $\mu_{d+1,\beta-\frac12}$.) It follows that
$$
p_d(\beta)
=
(d+2)\cdot \P\big[X_{d+2} \in [X_1,\ldots, X_{d+1}]\big]
=
(d+2)\cdot \P\big[\Pi_H(Y_{d+2})\in [\Pi_H(Y_1),\ldots, \Pi_H(Y_{d+1})]\big],
$$
which holds for any hyperplane $H$. By the formula of total probability, we may choose $H:= U^\perp$, where $U$ is uniform on $\bS^d$ and independent of $Y_1,\ldots, Y_{d+2}$, and we get
%
\begin{align*}
p_d(\beta)&=(d+2)\cdot \bE\big[\P\big[\Pi_{U^{\perp}}(Y_{d+2})\in [\Pi_{U^{\perp}}(Y_1),\ldots, \Pi_{U^{\perp}}(Y_{d+1})]\,|\, Y_1,\ldots, Y_{d+2}\big]\big]\\
& =2(d+2)\cdot  \E \alpha \big(\pos (Y_1-Y_{d+2}, \ldots, Y_{d+1} - Y_{d+2})\big)\\
&=
2 \cdot \bJ_{d+2,1} \left(\beta - \frac 12\right),
\end{align*}
where the second equality follows by Lemma \ref{lm:angles_probabilities}, while the last equality is a consequence of~\eqref{eq:Jbeta}.
Applying~\eqref{eq:J_nk_integral} with $n=d+2$ and $\alpha = 2\beta + d$ completes the proof in the beta case under the assumption $\beta \geq - \frac 12$.

To show that~\eqref{eq:sylvester_beta_2} holds for $\beta \geq -1$ and $d\geq 2$ we argue by analytic continuation. The function $p_d(\beta)$, originally defined for real $\beta>-1$, admits analytic continuation to the half-plane $\{\Re \beta >-1\}$. Indeed, we can write $p_d(\beta) = (d+2) \int_{D} f_{d,\beta}(x_1) \ldots f_{d,\beta} (x_{d+2}) \dint x_1 \dots \dint x_{d+2}$, where $D = \{(x_1,\dots, x_{d+2})\in (\bB^d)^{d+2}: x_{d+2} \in [x_1,\ldots, x_{d+1}]\}$ and then apply Lemma~4.3 in~\cite{beta_polytopes}. On the other hand, one can easily check that  the double integral  in~\eqref{eq:sylvester_beta_2} defines an analytic function on $\{\Re \beta > -\frac d2 - \frac 1{d+2}\}$. By the identity theorem for analytic functions (see, e.g., \cite[p.~125]{freitag_busam_book}), formula~\eqref{eq:sylvester_beta_2} for $p_d(\beta)$ continues to hold for all $\beta >-1$, provided $d\geq 2$, and the case $\beta = -1$ follows by continuity.

The proof in the beta prime case is similar if one starts with $\tilde Y_1,\ldots, \tilde Y_{d+2}$ that are i.i.d.\ in $\R^{d+1}$ and distributed according to $\tilde \mu_{d+1,\beta + \frac 12}$. Then, the projections of these points to any hyperplane in $\R^{d+1}$ are beta prime distributed with parameter $\beta$ (see, e.g., \cite[Lemma~4.4]{beta_polytopes_temesvari} or~\cite[Lemma~3.1]{beta_polytopes}). Note that, in contrast to the beta case, no additional restrictions on $\beta$ appear since $\beta + \frac 12 > \frac{d+1}{2}$. It remains to apply \eqref{eq:Jbetaprime} instead of~\eqref{eq:Jbeta} and~\eqref{eq:J_nk_tilde_integral} with $n=d+2$ and $\alpha = 2\beta - d$ instead of~\eqref{eq:J_nk_integral}. Since~\eqref{eq:J_nk_tilde_integral} requires $\alpha n >1$ we have to assume $\beta > \frac{d}2 + \frac 1 {2(d+2)}$.
\end{proof}

\begin{remark}
Proofs of both, Theorem \ref{theo:sylvester_normal} and Theorem \ref{theo:sylvester_beta}, rely on Lemma~\ref{lm:angles_probabilities} and the fact that standard normal, beta and beta prime distributions are rotationally invariant. Two additional crucial ingredients are the invariance of the above families of distributions under projections and the formulas for the expected solid angles of the corresponding simplices.
If some distribution $\mu$ in $\R^d$ can be represented as a projection of a rotationally distribution $\nu$ in $\R^{d+1}$, the probability $p_d(\mu)$ can be written as twice the angle of a $(d+1)$-dimensional random simplex with i.i.d.\ vertices distributed according to $\nu$. However, computing the expected solid angles of a random simplex is in general a hard problem.
\end{remark}



\section*{Acknowledgement}
Both authors were supported by the German Research Foundation under Germany's Excellence Strategy  EXC 2044 -- 390685587, \textit{Mathematics M\"unster: Dynamics - Geometry - Structure} and by the DFG priority program SPP 2265 \textit{Random Geometric Systems}.

\bibliography{sylvester_for_beta_bib_v4}
\bibliographystyle{plainnat}

\end{document}